\documentclass[10pt,a4paper]{amsart}
\usepackage[margin=1.0 in]{geometry}

\usepackage{amsmath,amssymb}
\usepackage[english]{babel}
\usepackage{url}
\usepackage{tikz,pgf}
\usetikzlibrary{calc}
\usetikzlibrary{decorations.markings}
\usetikzlibrary{shapes.geometric}
\usetikzlibrary{patterns}
\usetikzlibrary{intersections}
\usetikzlibrary{arrows.meta}
\usepackage[all]{xy}
\usepackage{tikz-cd}
\usepackage[colorinlistoftodos,prependcaption,textsize=tiny]{todonotes}
\usepackage{amsmath}
\usetikzlibrary{arrows,positioning,calc,patterns}
\usepackage{tikz-cd}
\usepackage{comment}
\usepackage{hyperref}
\hypersetup{
	colorlinks,
	linkcolor={red!50!black},
	citecolor={blue!50!black},
	urlcolor={blue!80!black},
	breaklinks=true
}
\usepackage{cleveref}
\tikzstyle{vertex}=[circle, draw=white, fill=black, inner sep=0pt, minimum size=4pt]
\tikzstyle{labelsty}=[font=\scriptsize]
\colorlet{ecol}{black!50!white}
\tikzstyle{edge}=[ecol,line width=1.5pt]

\theoremstyle{plain}
\newtheorem{lemma}{Lemma}[section]
\newtheorem{proposition}[lemma]{Proposition}
\newtheorem{theorem}[lemma]{Theorem}

\newtheorem*{theorem-nonum}{Theorem}

\theoremstyle{remark}
\newtheorem{remark}[lemma]{Remark}
\theoremstyle{definition}

\newtheorem{definition}[lemma]{Definition}

\newcommand{\R}{\mathbb{R}}

\title{A Geometric Criterion for Degeneracy in the Elekes-Szab\'{o} Theorem}

\author[M. Makhul]{Mehdi Makhul}
\address{Johann Radon Institute for Computational and Applied Mathematics (RICAM), Linz, Austria}
\email{makhul85@gmail.com}

\begin{document}

\begin{abstract}
The Elekes-Szab\'{o} theorem establishes that, under the usual finite-fibre projection hypothesis, an algebraic hypersurface $Z(F)$ contains few grid points unless it exhibits a specific local group-related structure. Identifying this structure directly from the defining equation is a challenging problem in combinatorial geometry. Our first main result
(Theorem~\ref{thm: regularity criterion}) gives a local differential characterization of this degeneracy for implicitly defined real analytic hypersurfaces, in terms of a separated factorization of the partial derivatives of $F$ along $Z(F)$.

We then develop a geometric framework for \textit{boundary varieties}, defined by the vanishing of partial derivatives along $Z(F)$. Theorem~\ref{thm: bounadry variety} shows that, when the separated factorization extends across a neighbourhood of the boundary, every local irreducible component of a boundary variety is contained in a coordinate slice. Proposition~\ref{prop:boundary-continuation} extends this conclusion to hypersurfaces which are degenerate on a dense coordinate-regular locus. This gives a geometric constraint on the loci where $Z(F)$ becomes tangent to coordinate directions.
		
As an application, we study configurations formed by $d$ coordinate-grid hyperplane families together with a one-parameter polynomial family of hyperspheres in $\mathbb{R}^d$. If one chooses $n$ members from each of these $d+1$ families and obtains $\Omega(n^{d-\eta})$ common incidence points, for a
suitable constant $\eta>0$, then the hypersphere family is forced to have a very restricted form: it is concentric in dimensions $d\geq 3$, and in dimension $2$ it is either concentric or consists of fixed-radius circles whose centres lie on a line parallel to a coordinate axis.

We also generalize the pinned distance problem initiated by Elekes and Szab\'{o} for three points in the plane to $d+1$ points in $\mathbb{R}^d$. More precisely, in Theorem~\ref{thm: pin distance} we prove that if $d+1$ families of hyperspheres centred at fixed points determine $\Omega(n^{d-\eta})$ points, for a suitable constant $\eta>0$, each lying on one hypersphere from each family, then the centres must be affinely dependent.
\end{abstract}
	
\maketitle
	
\section{Introduction}	
	
Throughout this paper, we write $Y=\Omega(X)$ if there exists a constant $c>0$, independent of the asymptotic parameter, such that $Y\geq cX$.

Finding upper bounds for the intersection of algebraic varieties with Cartesian grids is commonly referred to as the \emph{Elekes-Szab\'{o} problem}. In particular, suppose that $Z(F)\subset\R^d$ is an irreducible algebraic hypersurface of degree $r$, and assume that the projection of $Z(F)$ to any $d-1$ coordinates is finite-to-one. Then a higher-dimensional version of the Elekes-Szab\'{o} theorem asserts that exactly one of the following alternatives holds:
\begin{itemize}
	\item There exist a constant $c=c(d, r)>0$ and $\eta_0>0$ such that, for every Cartesian grid
	$A_1\times \cdots \times A_d$ with $|A_i|=n$, we have
	\[
    |Z(F)\cap(A_1\times\cdots\times A_d)| \le c n^{d-1-\eta_0}.
	\]
	\item There exist nonempty open sets $U_i\subset\mathbb R$, a common open interval $J\subset\mathbb R$ containing $0$, and analytic bijections $f_i:U_i\to J$ with analytic inverses such that, on $U_1\times\cdots\times U_d$,
	\begin{equation}\label{eq: ES}
	F(x_1,\ldots,x_d)=0
	\quad\Longleftrightarrow\quad
	f_1(x_1)+\cdots+f_d(x_d)=0.
	\end{equation}
\end{itemize}

We refer to the latter as the \emph{local group-structure}, or \emph{degenerate}, alternative. It is important that this conclusion is local: it is obtained on some nonempty open product set, not a priori at every coordinate-regular point of the hypersurface. In the applications below, Lemma~\ref{lem:propagation-degeneracy} will allow us to pass from this local conclusion to degeneracy in the sense of Definition~\ref{def:degenerate_polynomial}. The study of this problem was initiated by Elekes and Szab\'{o}~\cite{ES1}, who proved the first such dichotomy in the case $d=3$. The theorem was subsequently extended in several directions, including to the complex setting and to the four-dimensional case, using incidence-geometric methods~\cite{RSZ, RSZ1}. Bays and Breuillard~\cite{BB} obtained a similar result for an arbitrary number of variables, though without an explicit exponent and with a different description of the exceptional case. A quantitative higher-dimensional framework was later obtained by Chernikov, Peterzil, and Starchenko~\cite{CPS}.

The Elekes-Szab\'{o} theorem has seen a wide range of applications to both additive combinatorics and discrete geometry. Examples include the distinct-distances problem for points on two lines \cite{SSS}, distinct distances determined by points on algebraic curves in the plane \cite{PZ}, and bounds for triple intersection points determined by three families of unit circles \cite{Elekes, RSS}.

In another direction, several works have focused on improving the exponent $\eta_0$ appearing in the first alternative. In~\cite{MRWZ}, the authors constructed a polynomial $F$ in $d$ variables, together with sets $A_1,\dots,A_d$, each of cardinality $n$, such that
\[
|Z(F) \cap (A_1 \times \dots \times A_d)| = \Omega (n^{d-1-\frac{1}{2}}).
\]
In the case $d=3$, better bounds are known under additional assumptions or by different methods: an improved exponent was obtained in~\cite{MRSW} assuming the Uniformity Conjecture, while~\cite{SZ} obtained a further improvement using the proximity method.

While the Elekes-Szab\'{o} theorem establishes that such degeneracy is the only path to near-maximal incidence patterns between algebraic hypersurfaces and Cartesian grids, identifying this structure directly from the defining equation remains a significant challenge. Differential criteria are known in the setting of explicit polynomial functions. For example, Raz and Shem-Tov~\cite[Lemma~2.3]{RST} consider a polynomial $f(x_1,\ldots,x_d)$ satisfying a system of the form
\[
\frac{f_{x_1}}{r_1(x_1)}=\cdots=\frac{f_{x_d}}{r_d(x_d)},
\]
where the $r_i$ are one-variable rational functions, and show that this forces $f$ to have an additive or multiplicative separated polynomial form. Their setting concerns explicit polynomial functions, or equivalently their graph hypersurfaces. Here our aim is instead to give a local analytic criterion directly for a general implicitly defined hypersurface
\[
F(x_1,\ldots,x_d)=0.
\]
Thus the central problem addressed in this paper is the following: given a smooth hypersurface $Z(F)$ in $\R^d$, how can one determine directly from its defining equation whether it is degenerate?

A related geometric viewpoint appears in web geometry, where boundary curves are used to study regularity of curvilinear three-webs; see, for example, Shelekhov~\cite{Shelekhov}. Although that work belongs to the local theory of three-webs, it suggests that tangency loci can be useful in detecting hidden geometric structure. In the present paper we use this viewpoint in the Elekes-Szab\'{o} degeneracy problem for implicit analytic hypersurfaces.

Our first result gives a local differential criterion for this problem (Theorem~\ref{thm: regularity criterion}). It may be viewed as an implicit analytic analogue of the differential characterizations described above: degeneracy is equivalent to a separated factorization of the partial derivatives of $F$ along $Z(F)$ into a common analytic term and one-variable analytic factors.

We then introduce boundary varieties, which record the points where $Z(F)$ becomes tangent to coordinate directions. Theorem~\ref{thm: bounadry variety} shows that, when the separated factorization extends across a neighbourhood of the boundary, every local irreducible component of a boundary variety is contained in a coordinate slice. Proposition~\ref{prop:boundary-continuation} gives the corresponding conclusion when degeneracy is known only on a dense coordinate-regular locus. Thus Theorem~\ref{thm: regularity criterion} provides the differential characterization, while Theorem~\ref{thm: bounadry variety} and
Proposition~\ref{prop:boundary-continuation} provide the geometric boundary obstruction used in the applications below.

We use the boundary variety framework in two directions. The first concerns one-parameter families of hyperspheres. We show that if such a family has near-maximal intersection with Cartesian grids, then the family must be highly restricted. More precisely, for one-parameter polynomial families of hyperspheres, we prove that near-maximal intersections with Cartesian grids force the family to have one of two explicit forms: If $d\ge 3$, then the hyperspheres must be concentric. If $d=2$, then there are exactly two possible forms: either the circles are concentric, or they have fixed radius and their centres lie on a line parallel to one of the coordinate axes.

Our second application concerns a higher-dimensional version of the pinned distance problem initiated by Elekes and Szab\'{o}~\cite[Theorem $32$]{ES1}. Elekes and Szab\'{o} studied three families of circles in the plane, centred at three fixed points. Take $n$ circles from each family, and they showed that the existence of many triple-rich points (points lying on one circle from each family) forces the three centres to be collinear. We generalize this construction to an arbitrary dimension. More precisely, we consider $d+1$ families of hyperspheres in~$\mathbb{R}^d$, each family centred at a fixed point~$p_i$. We take $n$ hyperspheres from each family and prove that if these families determine sufficiently many $(d+1)$-rich points, then the centres $p_1,\dots,p_{d+1}$ must be affinely dependent.

\section{Statement of the main results}\label{sec: main results}

We now state the main results of the paper. The first step is to make precise the special local form appearing in the Elekes-Szab\'{o} alternative~\eqref{eq: ES}. We use the following definition throughout the paper.

\begin{definition}\label{def:coordinate-regular}
Let $\Omega\subset\mathbb R^d$ be open, let $F\colon\Omega\to\mathbb R$ be real analytic, and let $p\in Z(F):=\{x\in\Omega:F(x)=0\}$. We say that $p$ is a \emph{coordinate regular point} of $Z(F)$ if
\[
F_{x_i}(p)\neq0 \qquad\text{for every }i=1,\ldots,d.
\]
\end{definition}

\begin{definition}\label{def:degenerate_polynomial}
Let $\Omega\subset\R^d$ be open and let $F\colon\Omega\to\R$ be real analytic. We say that $F$ is \emph{locally degenerate} at a coordinate regular point $p=(a_1,\dots,a_d)\in Z(F)$ if there exist open intervals $I_i\subset\mathbb R$ containing $a_i$, with $U=I_1\times\cdots\times I_d\subset\Omega$, and analytic isomorphisms $f_i:I_i\to J$ onto a common open interval $J$ containing $0$ such that on $U$,
\[
F(x_1,\dots,x_d)=0 \;\Longleftrightarrow\; f_1(x_1)+\cdots+f_d(x_d)=0.
\]
We say that $F$ is \emph{degenerate} if it is locally degenerate at every coordinate regular point of $Z(F)$.
\end{definition}

The first main result gives a local differential characterization of this notion, in terms of a separated factorization of the partial derivatives of $F$ along its zero set.

\begin{theorem}\label{thm: regularity criterion}
Let $F(x_1,\ldots,x_d)$ be a locally irreducible real analytic function, and let $p=(a_1,\ldots,a_d)$ be a coordinate regular point of $Z(F)$. If $F$ is locally degenerate at $p$, then there exist a neighbourhood $U$ of $p$, a real analytic function $G(x_1,\ldots,x_d)$ on $U$, and two sets of real analytic single-variable functions
\[
\{\alpha_{1,j}(x_j)\}_{j=1}^d \qquad\text{and}\qquad \{\alpha_{2,j}(x_j)\}_{j=1}^d
\]
such that
\[
F_{x_i}=G(x_1,\ldots,x_d)\alpha_{2,i}(x_i) \prod_{j\neq i}\alpha_{1,j}(x_j)
\]
on $U\cap Z(F)$, for every $i=1,\ldots,d$. Conversely, if the above factorization holds in a neighbourhood of a coordinate regular point  $p$, then $F$ is locally degenerate at $p$.
\end{theorem}

%
Theorem~\ref{thm: regularity criterion} gives a differential criterion for degeneracy. For the applications below, we will use a geometric consequence of this criterion, formulated in terms of the following auxiliary varieties.

Let $V= Z(F) \subset \R^d$ be smooth real analytic hypersurface defined by the zero locus of a real analytic function $F(x_1,\dots, x_d)$. The \emph{coordinate regular locus} of $V$ is
\[
V^{\circ} = \left\{ p\in V_{\mathrm{sm}}: F_{x_i}(p)\neq 0 \text{ for every } i=1,\ldots,d \right\}.
\]
We associate $d$ families of varieties with $V$ by considering the intersections with coordinate hyperplanes:
\[
V_i(c):= V \cap \left\{x_i=c \right\}.
\]
We call $V_i(c)$ a coordinate slice of $V$.

For each $i \in \left\{1, \dots, d \right\}$ we define the $i$-th \emph{boundary variety} of $V$, denoted by $\Sigma_i(V)$, as the subset of $V_{\mathrm{sm}}$ where the $i$-th coordinate direction is tangent to the hypersurface:
\begin{equation}\label{eq: boundary variety}
\Sigma_i(V)=\left\{p\in V_{\mathrm{sm}}: \quad    \frac{\partial F}{\partial x_i}(p)=0 \right\}.
\end{equation}
Equivalently, 
\[
\Sigma_i(V)=V_{\mathrm{sm}} \cap Z(F_{x_i}).
\]
If $V$ is smooth, then this is given by
\[
\Sigma_i(V)=Z(F) \cap Z(F_{x_i}).
\]
The following result shows that, under the separated derivative factorization, boundary varieties are forced to lie locally inside coordinate slices.

\begin{theorem} \label{thm: bounadry variety}
Let $d \ge 2$, let $F(x_1,\ldots,x_d)$ be a real analytic function, and let $V=Z(F)$, and let $U$ be a neighbourhood of a point $p\in V$. Assume that 
\[
\nabla F(q) \not=0 \quad \text{for each}\quad q\in U \cap V.
\]
Suppose that, on $U\cap V$, the partial derivatives of $F$ admit a separated factorization
\[
F_{x_i} = G(x_1,\ldots,x_d)\alpha_{2,i}(x_i) \prod_{j\neq i}\alpha_{1,j}(x_j), \qquad i=1,\ldots,d,
\]
where $G$ is analytic on $U$, and the functions $\alpha_{1,j},\alpha_{2,j}$ are one-variable analytic functions which are not identically zero on the chosen coordinate neighbourhoods. Then every local irreducible component of the boundary variety 
\[ 
\Sigma_i(V)\cap U = \{q\in V\cap U:F_{x_i}(q)=0\} 
\] 
is contained in a coordinate slice 
\[ 
V_j(c)=V\cap\{x_j=c\} 
\] 
for some $j\in\{1,\ldots,d\}$ and some constant $c$.
\end{theorem}
The following proposition extends the geometric conclusion of Theorem~\ref{thm: bounadry variety} to hypersurfaces which are degenerate on a dense coordinate-regular locus.
\begin{proposition}\label{prop:boundary-continuation}
Let $d \ge 2$, let $F(x_1,\ldots,x_d)$ be a real analytic function, let $V=Z(F)$, and let $U$ be a neighbourhood of a point $p\in V$. Assume that
\[
\nabla F(q)\neq 0 \qquad\text{for every }q\in U\cap V.
\]
Let
\[
V^\circ=\left\{q\in U\cap V: F_{x_j}(q)\neq 0 \text{ for every }j=1,\ldots,d \right\}
\]
be the coordinate-regular locus in $U\cap V$. Suppose that $V^\circ$ is dense in $U\cap V$ and that $F$ is locally degenerate at every point of $V^\circ$. Then, for every $i\in\{1,\ldots,d\}$, every local irreducible component of
\[
\Sigma_i(V)\cap U=\{q\in U\cap V:F_{x_i}(q)=0\}
\]
is contained in a coordinate slice
\[
V_j(c)=V\cap\{x_j=c\}
\]
for some $j\in\{1,\ldots,d\}$ and some constant $c\in\mathbb R$.
\end{proposition}

\begin{remark} \label{rem:local_nature}
It is important to emphasize that the containment of a boundary variety in a coordinate slice is a local property. A boundary variety $\Sigma_i(V)$ may globally have several analytic components, and different local components may be contained in coordinate slices corresponding to different coordinates. Theorem \ref{thm: bounadry variety} asserts that each local irreducible component of $\Sigma_i(V)$ is contained in some coordinate slice 
\[ 
V_j(c)=V\cap\{x_j=c\} 
\]
for some $j\in\{1,\ldots,d\}$ and some constant $c$. The index $j$ need not be equal to $i$.
\end{remark}	

We now turn to the applications of the boundary variety framework. The first application concerns one-parameter families of hyperspheres and gives a classification of those families which can have near-maximal intersections with Cartesian grids. The second application concerns a higher-dimensional pinned distance problem, where the same framework is used to show that if $d+1$ fixed-centre families of hyperspheres determine many points, each lying on one hypersphere from every family, then the centres must be affinely dependent.

A one-parameter \emph{polynomial family} of hyperspheres in $\R^d$ is a family
\[
\mathcal P=\{S_t:t\in I\},
\]
where $I\subset\mathbb R$ is an open interval and
\[
S_t= \left\{ \bar x\in\mathbb R^d: \sum_{i=1}^d (x_i-a_i(t))^2=R^2(t) \right\},
\]
where $a_i(t), R(t) \in \R[t]$ and
\[
R(t) >0 \qquad \text{for each} t \in I.
\] 
We associate an incidence hypersurface to $\mathcal{P}$ given by:
\[
Q_{\mathcal P} =\{(\bar x,t)\in\mathbb R^d\times I:\bar x\in S_t\}.
\]
Equivalently,
\[
Q_{\mathcal P}=Z(F)\subset \mathbb R^d\times I,
\]
where
\begin{equation}\label{eq: family of hyperspheres}
	F(\bar x,t)=\sum_{i=1}^d (x_i-a_i(t))^2-R^2(t).
\end{equation}
We say that the family $\mathcal P$ is \emph{regular} if $Q_{\mathcal P}$ is smooth and the projection of $Q_{\mathcal P}$ onto any $d$ of the $d+1$ coordinates is
finite-to-one. Then we have:
\begin{theorem}\label{thm: pencil of spheres}
Let $\mathcal{P}=\{S_t:t\in I\}$ be a regular one-parameter polynomial family of hyperspheres in $\mathbb{R}^d$. Let $A_1,\dots,A_d\subset \mathbb{R}$ and $T_0\subset I$ be finite sets satisfying
\[
|A_1|=\cdots=|A_d|=|T_0|=n.
\]
Let $\eta_0>0$ be the exponent supplied by Theorem~\ref{thm:ES-general} for $s=d+1$, and set
\[
\eta=\frac{\eta_0}{2}.
\]
Suppose that
\[
|Q_{\mathcal{P}} \cap (A_1\times\cdots\times A_d\times T_0)| =\Omega(n^{d-\eta}).
\] 
Then $\mathcal{P}$ must satisfy one of the following:
\begin{itemize}
	\item $\mathcal{P}$ is a family of concentric hyperspheres;
	\item $d=2$, and $\mathcal{P}$ is a family of circles of fixed radius whose centres lie on a line parallel to one of the coordinate axes.
\end{itemize}
\end{theorem}	
The Elekes-Szab\'{o} theorem places any such extremal family in the local group-structure alternative. Lemma~\ref{lem:propagation-degeneracy} upgrades this local conclusion to degeneracy, after which the boundary variety framework gives the classification.

\begin{remark}
	The Elekes-Szab\'{o} theorem used here applies to semialgebraic
	relations. Thus, for the incidence application, we consider polynomial
	families of hyperspheres, so that the associated incidence relation
	$Q_{\mathcal P}$ is semialgebraic of bounded description complexity.
	The degeneracy analysis itself, however, does not require this
	algebraic assumption. We prove the classification under the weaker
	assumption that the centre curve and the radius function are real
	analytic. This analytic version may also be useful in the study of
	regular $(d+1)$-webs formed by the $d$ coordinate hyperplane
	foliations together with a one-parameter family of hyperspheres.
\end{remark}


We next consider a different incidence problem for hyperspheres, in which the centres are fixed and the radii vary. This gives a higher-dimensional version of the pinned distance problem of Elekes and Szab\'{o}.

\begin{theorem}\label{thm: pin distance}
Let $d\geq2$, let $p_1,\ldots,p_{d+1}\in\mathbb R^d$, and, for each $i=1,\ldots,d+1$, let $\mathcal S_i$ be a family of $n$ hyperspheres centred at~$p_i$. Let $\eta_0>0$ be the exponent supplied by Theorem~\ref{thm:ES-general} for~$s=d+1$, and set $\eta=\frac{\eta_0}{2}$. Suppose that the number of points in $\mathbb R^d$ which lie on one hypersphere from each of the families
\[
\mathcal S_1,\ldots,\mathcal S_{d+1}
\]
is $\Omega\!\left(n^{d-\eta}\right)$. Then the points
\[
p_1,\ldots,p_{d+1}
\]
are affinely dependent. Equivalently, they lie in a common affine hyperplane.
\end{theorem}

\section{Proof of Theorem \ref{thm: regularity criterion}}\label{sec: proof of regularity criterion}

\begin{proof}
Assume that $F$ is locally degenerate at the coordinate-regular point $p=(a_1,\dots, a_d)$ in the zero locus $Z(F)$. By Definition~\ref{def:degenerate_polynomial} there exist analytic isomorphisms $f_i$ such that, locally near $p$ we have
\[
F(x_1,\ldots,x_d)=0 \quad\Longleftrightarrow\quad f_1(x_1)+\cdots+f_d(x_d)=0.
\]
Since $p$ is coordinate regular, we have $F_{x_d}(p)\neq 0$. Thus, by the implicit function theorem, $Z(F)$ can be written locally as the graph of an
analytic function
\begin{equation}\label{eq: first additive}
x_d=\psi(x_1,\ldots,x_{d-1}).
\end{equation}
On the other hand, from the local additive form of $F$ we have
\[
f_d(x_d)=-(f_1(x_1)+\dots+f_{d-1}(x_{d-1})).
\]
Set $u=f_1(x_1)+\dots+f_{d-1}(x_{d-1})$, since $f_d$ is an analytic isomorphism, we can write
\begin{equation}\label{eq: second additive}
x_d=\phi(u),\quad \text{where} \quad \phi(u)=f_d^{-1}(-u).
\end{equation}
Therefore, from equalities \eqref{eq: first additive} and \eqref{eq: second additive} on the hypersurface $Z(F)$ we have
\[
\psi(x_1,\dots,x_{d-1})=\phi(f_1(x_1)+\dots+f_{d-1}(x_{d-1})).
\]
Now differentiate with respect to $x_i$, where $i=1,\dots, d-1$ . We get
\[
\frac{\partial \psi}{\partial x_i}= \phi'(u) f_i'(x_i).
\]
On the other hand on $Z(F)$ we have 
\[
F(x_1,\dots, x_{d-1}, \psi(x_1,\dots, x_{d-1}))=0.
\]
Differentiating this identity with respect to $x_i$ gives
\[
\frac{\partial F}{\partial x_i} + \frac{\partial F}{\partial x_d} \cdot \frac{\partial \psi}{\partial x_i}=0
\]
Substituting the formula for $\frac{\partial \psi}{\partial x_i}$ we obtain:
\[
\frac{\partial F}{\partial x_i}=- \frac{\partial F}{\partial x_d}\phi'(u)f_i'(x_i), \quad i=1,\dots, d-1.
\]
Now we consider the derivative with respect to \(x_d\). Since
\[
\phi(u)=f_d^{-1}(-u),
\]
we have
\[
f_d(\phi(u))=-u.
\]
Differentiating with respect to \(u\), we obtain
\[
f_d'(\phi(u))\phi'(u)=-1.
\]
On \(Z(F)\), we have \(x_d=\phi(u)\), and hence
\[
f_d'(x_d)\phi'(u)=-1.
\]
Therefore,
\[
-\phi'(u)f_d'(x_d)=1.
\]
Multiplying both sides by \(F_{x_d}\), we get
\[
F_{x_d}
=
-F_{x_d}\phi'(u)f_d'(x_d).
\]
If we define
\[
G(x_1,\ldots,x_d)
=
-F_{x_d}(x_1,\ldots,x_d)\phi'(u),
\]
then this becomes
\[
F_{x_d}=G(x_1,\ldots,x_d)f_d'(x_d).
\]
Now for $i,j = 1,\dots, d$, define
\[
\alpha_{2,i}(x_i)=f'_i(x_i), \quad \alpha_{1,j}(x_j)=1.
\]
Then, for each $i\in\{1,\ldots,d\}$, we have
\[
\prod_{j\neq i}\alpha_{1,j}(x_j)=1.
\]
Therefore,
\[
G(x_1,\ldots,x_d)\alpha_{2,i}(x_i) \prod_{j\neq i}\alpha_{1,j}(x_j)= G(x_1,\ldots,x_d)f_i'(x_i).
\]
By the previous computation, the right-hand side is equal to $F_{x_i}$ on $Z(F)$. Hence, for every $i=1,\ldots,d$,
\[
F_{x_i}=G(x_1,\ldots,x_d)\alpha_{2,i}(x_i)\prod_{j\neq i}\alpha_{1,j}(x_j)
\]
on $Z(F)$. This is precisely the factorization required in the statement of the theorem.

Conversely, assume that
\[
F_{x_i}(p)\neq 0
\qquad\text{for every } i=1,\ldots,d,
\]
and suppose that in a neighbourhood $U$ of $p=(a_1,\ldots,a_d)\in Z(F)$, the partial derivatives satisfy
\[
F_{x_i}= G(x_1,\ldots,x_d)\alpha_{2,i}(x_i) \prod_{j\neq i}\alpha_{1,j}(x_j)
\]
on $U\cap Z(F)$, for every $i=1,\ldots,d$. Evaluating this identity at $p$, we obtain
\[
F_{x_i}(p)= G(p)\alpha_{2,i}(a_i) \prod_{j\neq i}\alpha_{1,j}(a_j).
\]
Since $F_{x_i}(p)\neq 0$ for every $i$, all factors appearing in these products
are nonzero. Hence
\[
G(p)\neq 0, \qquad \alpha_{1,i}(a_i)\neq 0, \qquad \alpha_{2,i}(a_i)\neq 0
\]
for every $i=1,\ldots,d$. After shrinking $U$, we may assume that $G$, the functions $\alpha_{1,i}$, and the functions $\alpha_{2,i}$ do not vanish on the relevant neighbourhoods. In particular, the quotients
\[
\frac{\alpha_{2,i}(x_i)}{\alpha_{1,i}(x_i)}
\]
are analytic and nonzero near $a_i$.

Define
\begin{equation}\label{eq: auxiliary function}
p_i(x_i) = \int_{a_i}^{x_i} \frac{\alpha_{2,i}(t)}{\alpha_{1,i}(t)}\,dt, \qquad i=1,\ldots,d.
\end{equation}
Then each $p_i$ is real analytic, and
\[
p_i'(a_i)= \frac{\alpha_{2,i}(a_i)}{\alpha_{1,i}(a_i)} \neq 0.
\]
Therefore, after shrinking the coordinate intervals if necessary, each $p_i$ is a local analytic isomorphism.

Since $F_{x_d}(p)\neq 0$, the implicit function theorem implies that $Z(F)$ can be written locally near $p$ as the graph of an analytic function
\[
x_d=\psi(x_1,\ldots,x_{d-1}).
\]
Define
\[
\Phi(x_1,\ldots,x_{d-1}) = p_1(x_1)+\cdots+p_{d-1}(x_{d-1}) + p_d(\psi(x_1,\ldots,x_{d-1})).
\]
\textbf{Claim that $\Phi$ is constant in neighbourhood of $(a_1,\ldots,a_{d-1})$:}

For $i=1,\ldots,d-1$, differentiating $\Phi$ with respect to $x_i$ gives
\[
\frac{\partial \Phi}{\partial x_i}= p_i'(x_i) + p_d'(\psi(x_1,\ldots,x_{d-1})) \frac{\partial \psi}{\partial x_i}.
\]
By the definition of $p_i$, we have
\[
p_i'(x_i)= \frac{\alpha_{2,i}(x_i)}{\alpha_{1,i}(x_i)}
\]
and, on the graph $x_d=\psi(x_1,\ldots,x_{d-1})$,
\[
p_d'(\psi(x_1,\ldots,x_{d-1}))= \frac{\alpha_{2,d}(x_d)}{\alpha_{1,d}(x_d)}.
\]
On the other hand,
\[
F(x_1,\ldots,x_{d-1},\psi(x_1,\ldots,x_{d-1}))=0.
\]
Differentiating this identity with respect to $x_i$, we obtain
\[
F_{x_i} +  F_{x_d}\frac{\partial \psi}{\partial x_i} = 0.
\]
Hence
\[
\frac{\partial \psi}{\partial x_i} = -\frac{F_{x_i}}{F_{x_d}}.
\]
Using the assumed factorization of the partial derivatives, this becomes
\[
\frac{\partial \psi}{\partial x_i}=-\frac{\alpha_{2,i}(x_i)\alpha_{1,d}(x_d)}{\alpha_{1,i}(x_i)\alpha_{2,d}(x_d)}.
\]
Therefore,
\[
\frac{\partial \Phi}{\partial x_i}=\frac{\alpha_{2,i}(x_i)}{\alpha_{1,i}(x_i)}+\frac{\alpha_{2,d}(x_d)}{\alpha_{1,d}(x_d)}\left(-\frac{\alpha_{2,i}(x_i)\alpha_{1,d}(x_d)
}{\alpha_{1,i}(x_i)\alpha_{2,d}(x_d)}\right)=0.
\]
Thus
\[
\frac{\partial \Phi}{\partial x_i}=0
\]
for every $i=1,\ldots,d-1$. Hence $\Phi$ is constant. Since $p_i(a_i)=0$ for every $i$, evaluating at $p$ shows that this constant is zero. After shrinking the coordinate intervals if necessary, the $p_i$ are analytic isomorphisms onto a common small interval $J$ containing $0$. Therefore, locally on $Z(F)$, we have
\begin{equation}\label{eq: small equation}
p_1(x_1)+\cdots+p_d(x_d)=0.
\end{equation}
It remains to prove the converse implication. Suppose that \eqref{eq: small equation} holds. We show that $(x_1,\ldots,x_d)\in Z(F)$. Notice that since $p_d$ is a local analytic isomorphism, the equation~\eqref{eq: small equation} can be solved uniquely for $x_d$ as an analytic function of $x_1,\ldots,x_{d-1}$. Namely, on $Z(F)$, this solution is precisely $x_d=\psi(x_1,\ldots,x_{d-1})$. Hence, after shrinking the neighbourhood if necessary,
\[
F(x_1,\ldots,x_d)=0 \quad\Longleftrightarrow\quad p_1(x_1)+\cdots+p_d(x_d)=0.
\]
This is the desired local additive form. Therefore $F$ is locally degenerate at $p$.
\end{proof}

\section{Proofs of the boundary variety results}
We first prove Theorem~\ref{thm: bounadry variety}, and then establish Proposition~\ref{prop:boundary-continuation}, which extends its geometric conclusion to the dense coordinate-regular locus setting.


\begin{proof}[Proof of Theorem~\ref{thm: bounadry variety}]
By assumption, on a neighbourhood $U$, the partial derivatives of $F$ admit the separated factorization
\[
F_{x_i} = G(x_1,\ldots,x_d)\alpha_{2,i}(x_i) \prod_{j\neq i}\alpha_{1,j}(x_j), \qquad i=1,\ldots,d.
\]
Let $\Gamma$ be a local irreducible analytic component of $\Sigma_i(V)\cap U$, and let $p\in \Gamma$. By definition of the boundary variety
\[
F_{x_i}(q)=0 \qquad\text{for every }q\in \Gamma.
\]
Hence, for every \(q=(q_1,\ldots,q_d)\in\Gamma\), we have
\[
G(q)\alpha_{2,i}(q_i)\prod_{j\neq i}\alpha_{1,j}(q_j)=0.
\]
We first observe that $G(q)\neq0$ for every $q\in\Gamma$. Indeed, if $G(q)=0$, then the factorization gives
\[
F_{x_1}(q)=\cdots=F_{x_d}(q)=0,
\]
which contradicts the smoothness of $V=Z(F)$ at $q$. Therefore,
\[
\alpha_{2,i}(q_i)\prod_{j\neq i}\alpha_{1,j}(q_j)=0 \qquad\text{for every }q\in\Gamma.
\]
It follows that
\[
\Gamma \subset \left\{\alpha_{2,i}(x_i)=0\right\} \cup \bigcup_{j\neq i}\left\{\alpha_{1,j}(x_j)=0\right\}.
\]
Each function appearing here is a one-variable analytic function. Since these functions are not identically zero on the chosen coordinate neighbourhoods, their zeros are locally isolated. Hence, after shrinking \(U\) if necessary, each set
\[
\{\alpha_{2,i}(x_i)=0\} \quad\text{or}\quad \{\alpha_{1,j}(x_j)=0\}
\]
is a finite union of coordinate hyperplanes, and after intersecting with $V$, it gives coordinate slices. Therefore $\Gamma$ is contained in a finite union of coordinate slices. Since $\Gamma$ is a local irreducible analytic component, it must be contained in one of these coordinate slices. Thus there exist an index $j\in\{1,\ldots,d\}$ and a constant $c$ such that
\[
\Gamma\subset V_j(c)=V\cap\{x_j=c\}.
\]
This proves the claim.
\end{proof}


\begin{proof}[Proof of Proposition~\ref{prop:boundary-continuation}]
Fix $i\in\{1,\ldots,d\}$, and let $\Gamma$ be a local irreducible component of $\Sigma_i(V)\cap U$. We prove that $\Gamma$ is contained in a coordinate slice.
	
If $\Gamma$ is zero-dimensional, then the conclusion is immediate. Thus, we may assume that $\Gamma$ has positive dimension.
	
We first dispose of the case $d=2$. In this case $V$ is locally a one-dimensional smooth real analytic manifold. Since $V^\circ$ is dense in $U\cap V$, the function $F_{x_i}|_V$ cannot vanish identically on a nonempty open subset of $V$. Hence its zero set is locally zero-dimensional. This contradicts the assumption that $\Gamma$ has positive dimension. Therefore the conclusion is immediate when $d=2$.
	
Assume from now on that $d\geq3$. Choose a smooth point $q$ of $\Gamma$. Since
\[
F_{x_i}(q)=0
\qquad \text{and} \qquad
\nabla F(q)\neq0,
\]
there exists an index $k\neq i$ such that
\[
F_{x_k}(q)\neq0.
\]
After shrinking the neighbourhood of $q$, we may assume that
\[
F_{x_k}\neq0
\]
throughout this neighbourhood. By the implicit function theorem, after shrinking the neighbourhood of $q$ if necessary, $V$ can be written locally as the graph
\[
x_k=\psi(x_1,\ldots,\widehat{x_k},\ldots,x_d),
\]
where $\psi$ is real analytic. Thus
\[
F(x_1,\ldots,x_{k-1},\psi,x_{k+1},\ldots,x_d)=0
\]
identically in the remaining $d-1$ variables. Differentiating with respect to $x_\ell$, for $\ell\neq k$, gives
\begin{equation}\label{eq:psi-partial}
\psi_{x_\ell}=-\frac{F_{x_\ell}}{F_{x_k}}.
\end{equation}
In particular, since $F_{x_i}=0$ on $\Gamma$,
\[
\psi_{x_i}=0
\qquad\text{on }\Gamma.
\]
We now use degeneracy on the coordinate-regular locus. Let $p\in V^\circ$ be a coordinate-regular point in the above neighbourhood. By assumption, $F$ is locally degenerate at $p$. Hence, after shrinking a neighbourhood of $p$ if necessary, there exist one-variable real analytic isomorphisms $f_1,\ldots,f_d$ such that
\[
F(x_1,\ldots,x_d)=0
\quad\Longleftrightarrow\quad
f_1(x_1)+\cdots+f_d(x_d)=0.
\]
Since the graph $x_k=\psi$ parametrizes $V$ locally, the latter relation holds identically on the graph; that is,
\[
\sum_{\ell\neq k}f_\ell(x_\ell)+f_k(\psi)=0.
\]
Differentiating this identity with respect to $x_\ell$, for $\ell\neq k$, we obtain
\[
f_\ell'(x_\ell)+f_k'(\psi)\psi_{x_\ell}=0,
\]
and therefore
\begin{equation}\label{eq:graph-derivative-separated}
\psi_{x_\ell}=-\frac{f_\ell'(x_\ell)}{f_k'(\psi)}.
\end{equation}
%
%
We distinguish two cases.

\medskip
\noindent
\textit{Case 1.}
Suppose that, for some $j\notin\{i,k\}$,
\[
\psi_{x_j}|_\Gamma\not\equiv0.
\]
Since $\psi_{x_j}|_\Gamma\not\equiv0$, after replacing $q$ by a nearby smooth point of $\Gamma$ inside the above graph neighbourhood, we may assume that
\[
\psi_{x_j}(q)\neq0.
\]
After shrinking the neighbourhood of $q$ once more, we may assume that
\[
\psi_{x_j}\neq0
\]
throughout it. Hence
\[
r=\frac{\psi_{x_i}}{\psi_{x_j}}
\]
is a real analytic function on this neighbourhood.
	
At every point of $V^\circ$, equation~\eqref{eq:graph-derivative-separated} gives
\begin{equation}\label{eq:ratio-separated}
r=\frac{f_i'(x_i)}{f_j'(x_j)}.
\end{equation}
Although the functions $f_\ell$ may depend on the chosen coordinate-regular neighbourhood, equation~\eqref{eq:ratio-separated} implies the following intrinsic differential identities for $r$. For every $\ell\notin\{i,j,k\}$,
\begin{equation}\label{eq:r-independent}
r_{x_\ell}=0,
\end{equation}
and
\begin{equation}\label{eq:r-mixed}
r\,r_{x_i x_j}-r_{x_i}r_{x_j}=0.
\end{equation}
Indeed, both identities follow immediately from the separated form in \eqref{eq:ratio-separated}.
	
The expressions in \eqref{eq:r-independent} and \eqref{eq:r-mixed} are real analytic. Since $V^\circ$ is dense in $U\cap V$, these identities extend to the whole graph neighbourhood.
	
It follows from \eqref{eq:r-independent} that $r$ depends only on $x_i$ and $x_j$. Thus, after shrinking to a product neighbourhood if necessary, we may write
\[
r=R(x_i,x_j)
\]
for a real analytic function $R$. Moreover,
\begin{equation}\label{eq:R-equation}
R R_{x_i x_j}-R_{x_i}R_{x_j}=0.
\end{equation}
The function $R$ is not identically zero. Indeed, if $R\equiv0$, then $\psi_{x_i}\equiv0$ on an open subset of $V$, and by \eqref{eq:psi-partial} we would have $F_{x_i}\equiv0$ there, contradicting the density of $V^\circ$.
	
Choose $(a,b)$ in the coordinate neighbourhood such that
\[
R(a,b)\neq0.
\]
Near $(a,b)$, equation \eqref{eq:R-equation} gives
\[
\frac{\partial^2}{\partial x_i\partial x_j}\log|R|=0.
\]
Since the mixed derivative of $\log|R|$ vanishes, we have locally
\[
\log|R(x_i,x_j)|+\log|R(a,b)|= \log|R(x_i,b)|+\log|R(a,x_j)|.
\]
Exponentiating, and using that $R$ has constant sign in a sufficiently small neighbourhood of $(a,b)$, gives
\[
R(x_i,x_j)R(a,b)=R(x_i,b)R(a,x_j).
\]
Both sides are real analytic functions of $(x_i,x_j)$. Therefore, by the real analytic identity theorem, the same identity holds throughout the connected coordinate neighbourhood. Setting
\[
A(x_i)=R(x_i,b),\qquad B(x_j)=\frac{R(a,x_j)}{R(a,b)},
\]
we obtain
\[
R(x_i,x_j)=A(x_i)B(x_j)
\]
throughout that neighbourhood. Neither $A$ nor $B$ is identically zero. Since they are one-variable real analytic functions, their zeros are locally isolated. After shrinking the neighbourhood if necessary,
\[
\{R=0\}
\]
is therefore a finite union of coordinate hyperplanes of the form
\[
x_i=c\qquad\text{or}\qquad x_j=c.
\]
On $\Gamma$ we have $\psi_{x_i}=0$, while $\psi_{x_j}\neq0$ in the chosen neighbourhood. Hence
\[
\Gamma\subset\{r=0\}=\{R=0\}
\]
locally near $q$. Since $\Gamma$ is irreducible, one of the coordinate hyperplanes occurring above contains a nonempty relatively open subset of $\Gamma$. Thus, for some $m\in\{i,j\}$ and some constant $c$,
\[
x_m-c
\]
vanishes on a nonempty relatively open subset of $\Gamma$. Since $x_m-c$ is real analytic and $\Gamma$ is irreducible, it vanishes identically on $\Gamma$. Consequently,
\[
\Gamma\subset\{x_m=c\}.
\]
\medskip
\noindent
\textit{Case 2.}
Suppose that
\[
\psi_{x_j}|_\Gamma\equiv0
\qquad
\text{for every }j\neq k.
\]
Since $\psi$ is a function of the $d-1$ variables
$x_1,\ldots,\widehat{x_k},\ldots,x_d$, we have
\[
d\psi
=
\sum_{j\neq k}\psi_{x_j}\,dx_j.
\]
Hence
\[
d\psi|_\Gamma=0,
\]
and therefore, on the smooth locus of $\Gamma$,
\[
d(\psi|_\Gamma)=0.
\]
Since $\Gamma$ has positive dimension, $\psi$ is locally constant
on each connected relatively open subset of the smooth locus of
$\Gamma$. Thus, for some constant $c$,
\[
\psi=c
\]
on such a subset.
	
The function $\psi-c$ is real analytic on the graph neighbourhood, and its restriction to the irreducible analytic set $\Gamma$ vanishes on a nonempty relatively open subset. Hence it vanishes identically on $\Gamma$. Therefore
\[
\psi\equiv c\qquad\text{on }\Gamma.
\]
Since $x_k=\psi$ on $V$, it follows that
\[
\Gamma\subset\{x_k=c\}.
\]
In both cases, $\Gamma$ is contained in a coordinate slice. Since $i$ and $\Gamma$ were arbitrary, the proof is complete.
\end{proof}

In the applications below, the Elekes-Szab\'{o} theorem of Chernikov, Peterzil and Starchenko gives local degeneracy only on a nonempty open subset of the hypersurface. The following lemma allows us to pass from this local conclusion to degeneracy at every coordinate-regular point of a connected smooth hypersurface.

\begin{lemma}\label{lem:propagation-degeneracy}
Let $d\geq3$, let $F$ be real analytic on an open set $\Omega\subset\mathbb R^d$, and let $V$ be a connected relatively open subset of $Z(F)\cap\Omega$ such that $\nabla F\neq0$ on $V$. Let
\[
V^\circ=\{p\in V:F_{x_i}(p)\neq0\text{ for every }i=1,\ldots,d\}.
\]
Suppose that $F$ is locally degenerate on a nonempty relatively open subset of $V^\circ$. Then $F$ is locally degenerate at every point of $V^\circ$.
\end{lemma}

\begin{proof}
Let $W\subset V^\circ$ be a nonempty relatively open set on which $F$ is locally degenerate. Fix the last coordinate $x_d$. At every point of $V^\circ$ we have $F_{x_d}\neq0$, so locally $V$ may be written as a graph
\[
x_d=\psi(x_1,\ldots,x_{d-1}).
\]
For $i=1,\ldots,d-1$, implicit differentiation gives
\[
\psi_{x_i}=-\frac{F_{x_i}}{F_{x_d}}.
\]
For $i\neq j$ in $\{1,\ldots,d-1\}$, define on $V^\circ$
\[
r_{ij}=\frac{\psi_{x_i}}{\psi_{x_j}}=\frac{F_{x_i}}{F_{x_j}},
\]
and define the tangent differential operators
\[
D_i=\frac{\partial}{\partial x_i}-\frac{F_{x_i}}{F_{x_d}}\frac{\partial}{\partial x_d},
\qquad i=1,\ldots,d-1.
\]
On a graph neighbourhood, $D_i$ is differentiation with respect to the graph coordinate $x_i$.

Since $F$ is locally degenerate on $W$, locally on $W$ we have
\[
f_1(x_1)+\cdots+f_d(x_d)=0
\]
for one-variable analytic local isomorphisms $f_i$. Differentiating on the graph gives
\[
f_i'(x_i)+f_d'(\psi)\psi_{x_i}=0,
\]
and hence
\[
r_{ij}=\frac{f_i'(x_i)}{f_j'(x_j)}.
\]
Consequently, on $W$,
\begin{equation}\label{eq:prop-independence}
D_k r_{ij}=0
\end{equation}
whenever $i,j,k$ are pairwise distinct, and
\begin{equation}\label{eq:prop-mixed}
r_{ij}D_iD_jr_{ij}-(D_i r_{ij})(D_j r_{ij})=0.
\end{equation}

The expressions in~\eqref{eq:prop-independence} and \eqref{eq:prop-mixed} are rational expressions in derivatives of $F$. Their denominators are products of coordinate partial derivatives of $F$, and hence are nonzero on $V^\circ$. Thus, after multiplying by a common denominator, each identity is equivalent on $V^\circ$ to the vanishing of a real analytic function on~$V$. These analytic functions vanish on the nonempty relatively open set~$W$. Since $V$ is connected,
the real analytic identity theorem implies that they vanish identically on $V$. Restricting again to $V^\circ$, where the denominators are nonzero, yields~\eqref{eq:prop-independence} and \eqref{eq:prop-mixed} throughout $V^\circ$.

Now let $p=(c_1,\ldots,c_d)\in V^\circ$ be arbitrary, and shrink to a graph neighbourhood of $p$ on which all $\psi_{x_i}$ are nonzero. We first show that there exist nonvanishing one-variable analytic functions $a_i(x_i)$ and a nonvanishing analytic function $\lambda$ such that
\begin{equation}\label{eq:common-lambda}
\psi_{x_i}=\lambda\,a_i(x_i),
\qquad i=1,\ldots,d-1.
\end{equation}

If $d=3$, there are two graph variables. Put $r=r_{12}$. Equation~\eqref{eq:prop-mixed} gives
\[
\frac{\partial^2}{\partial x_1\partial x_2}\log|r|=0,
\]
so locally
\[
r=\frac{a_1(x_1)}{a_2(x_2)}
\]
for nonvanishing one-variable analytic functions $a_1,a_2$. Then
\[
\lambda=\frac{\psi_{x_1}}{a_1(x_1)}
      =\frac{\psi_{x_2}}{a_2(x_2)}
\]
gives~\eqref{eq:common-lambda}.

Suppose now that $d\geq4$. For $i=1,\ldots,d-1$, set
\[
u_i=\log|\psi_{x_i}|.
\]
For each $k\in\{1,\ldots,d-1\}$, define
\[
b_k=\frac{\partial u_i}{\partial x_k}
\]
using any index $i\neq k$. This is well defined: if $i,j\neq k$, then~\eqref{eq:prop-independence} gives
\[
\frac{\partial}{\partial x_k}(u_i-u_j)=0.
\]
The one-form
\[
\beta=\sum_{k=1}^{d-1} b_k\,dx_k
\]
is closed. Indeed, if $k\neq \ell$, choose $m$ distinct from both $k$ and $\ell$. Then
\[
b_k=\frac{\partial u_m}{\partial x_k},
\qquad
b_\ell=\frac{\partial u_m}{\partial x_\ell},
\]
so
\[
\frac{\partial b_k}{\partial x_\ell}
=
\frac{\partial b_\ell}{\partial x_k}.
\]
After shrinking the neighbourhood, there is therefore an analytic function $L$ with $dL=\beta$. Set $\lambda=e^L$. For $k\neq i$ we then have
\[
\frac{\partial}{\partial x_k}
\log\left|\frac{\psi_{x_i}}{\lambda}\right|=0.
\]
Hence
\[
a_i(x_i):=\frac{\psi_{x_i}}{\lambda}
\]
depends only on $x_i$, and~\eqref{eq:common-lambda} follows.

Since mixed partial derivatives of $\psi$ commute, for $i\neq j$ equation~\eqref{eq:common-lambda} gives
\[
a_i(x_i)\lambda_{x_j}=a_j(x_j)\lambda_{x_i}.
\]
Equivalently,
\[
\lambda_{x_i}\psi_{x_j}-\lambda_{x_j}\psi_{x_i}=0,
\]
so
\[
d\lambda\wedge d\psi=0.
\]
Because $d\psi\neq0$ on the chosen neighbourhood, $\lambda$ is constant on the local level sets of $\psi$. Thus there is a nonvanishing one-variable analytic function $h$ such that
\[
\lambda=h(\psi).
\]
Therefore
\[
\psi_{x_i}=h(\psi)a_i(x_i).
\]

Choose one-variable analytic functions $f_i$, $i=1,\ldots,d-1$, and $f_d$ satisfying
\[
f_i'(x_i)=-a_i(x_i),
\qquad
f_d'(x_d)=\frac1{h(x_d)}.
\]
Then, on the graph $x_d=\psi(x_1,\ldots,x_{d-1})$,
\[
\frac{\partial}{\partial x_i}
\left(
 f_1(x_1)+\cdots+f_{d-1}(x_{d-1})+f_d(\psi)
\right)=0
\]
for every $i=1,\ldots,d-1$. Hence, after shrinking to a connected graph neighbourhood,
\[
f_1(x_1)+\cdots+f_d(x_d)=C
\]
on $V$. Subtracting $C$ from one of the functions, we may assume $C=0$.

At the point $p=(c_1,\ldots,c_d)$ we then have
\[
f_1(c_1)+\cdots+f_d(c_d)=0.
\]
Replacing each $f_i$ by $f_i-f_i(c_i)$ preserves the relation and makes $f_i(c_i)=0$ for every $i$. All derivatives $f_i'$ are nonzero, so after shrinking the coordinate intervals the $f_i$ are analytic isomorphisms onto a common small interval $J$ containing $0$.

Finally, the equation
\[
f_1(x_1)+\cdots+f_d(x_d)=0
\]
defines, by the implicit function theorem and the fact that $f_d'\neq0$, a unique local graph in the $x_d$-direction. This graph contains the local graph $V$, so the two coincide. Therefore
\[
F(x_1,\ldots,x_d)=0
\quad\Longleftrightarrow\quad
f_1(x_1)+\cdots+f_d(x_d)=0
\]
locally near $p$. Thus $F$ is locally degenerate at $p$. Since $p\in V^\circ$ was arbitrary, the conclusion follows.
\end{proof}

\section{Degeneracy of the family of analytic hyperspheres}\label{sec: distance phenomena}	
Recall from Section~\ref{sec: main results} that a one-parameter family of hyperspheres
\[
\mathcal P=\{S_t:t\in I\}
\]
is encoded by the incidence hypersurface
\[
Q_{\mathcal P}=Z(F)\subset\mathbb R^d\times I,
\]
where
\[
F(\bar x,t)= \sum_{i=1}^d (x_i-a_i(t))^2-R(t)^2.
\]
In Section~\ref{sec: main results}, we stated the incidence application for polynomial families, so that the associated incidence relation is semialgebraic and the Elekes-Szab\'{o} theorem applies. In the present section, however, the degeneracy analysis is local and analytic. We therefore allow the centre functions $a_i(t)$ and the radius function $R(t)$ to be real analytic on the interval $I$.

Our goal is to classify the analytic families for which the associated function $F(\bar x,t)$ is degenerate in the sense of Definition~\ref{def:degenerate_polynomial}. 

In this section the variables are
\[
x_1,\ldots,x_d,t.
\]
Thus, when applying Proposition~\ref{prop:boundary-continuation}, coordinate slices may be of the form
\[
x_j=\mathrm{const}
\]
or of the form
\[
t=\mathrm{const}.
\]

\begin{lemma}\label{lm: extra boundary variety}
Let $d \ge 2$, and let
\[
F(x_1,\dots,x_d,t)=\sum_{i=1}^{d}(x_i-a_i(t))^2-R^2(t)
\]
define a real-analytic family of spheres in \(\mathbb R^d\), where the functions $a_i(t)$ and $R(t)$ are real analytic on an interval $I$, and assume $R(t)>0$ for each $t \in I$. For each $i\in \{1,\dots,d\}$, define
\[
\Sigma_i:= \left\{(x_1,\dots, x_d,t)\in \mathbb R^{d+1}:\, F=0,\quad \frac{\partial F}{\partial x_i}=0 \right\}.
\]
Then for each $t\in I$, the fibre $\Sigma_i(t)$ is nonempty. Moreover, no local irreducible component of $\Sigma_i$ is contained in a slice
\[
\{t=\mathrm{const}\}.
\]
\end{lemma}

\begin{proof}
We have
\[
\frac{\partial F}{\partial x_i}=2(x_i-a_i(t)).
\]
Hence $\Sigma_i$ is given by
\[
x_i=a_i(t),
\]
together with
\[
\sum_{j\neq i}(x_j-a_j(t))^2=R^2(t).
\]
For fixed $t\in I$, this is the equation of a sphere of radius $R(t)$ in the $(d-1)$-dimensional affine space
\[
\{x_i=a_i(t)\}.
\]
Since $R(t)>0$, this fibre is nonempty. In fact, it is a sphere of dimension $d-2$.
	
Moreover, this description varies analytically with $t$. Thus, near every point of $\Sigma_i$, the parameter $t$ varies along the local component. Therefore no local irreducible component of $\Sigma_i$ can be contained in a slice
\[
\{t=\mathrm{const}\}.
\]
\end{proof}

\begin{theorem}\label{thm:analytic-hypersphere-degeneracy}
Let $d\geq2$ and let $I\subset\mathbb R$ be an open interval. Let $\gamma \colon I \rightarrow \R^d$, $\gamma(t)=(a_1(t),\ldots,a_d(t))$, and $R \colon I \rightarrow \mathbb R$ be real-analytic functions such that $R(t)>0$ for every $t\in I$. Define
\[
F(\bar{x},t)=\sum_{i=1}^d (x_i-a_i(t))^2-R(t)^2.
\]
Then:
\begin{itemize}
	\item If $d=2$, then $F$ is degenerate in the sense of Definition~\ref{def:degenerate_polynomial} if and only if either the image of $\gamma$ is contained in a line parallel to one of the coordinate axes and $R$ is constant, or the family is concentric, equivalently, $\gamma$ is constant.
	\item If $d\geq3$, then $F$ is degenerate in the sense of Definition~\ref{def:degenerate_polynomial} if and only if the family is concentric, equivalently, $\gamma(t)\equiv a_0$ for some fixed point $a_0\in\mathbb R^d$.
	\end{itemize}
\end{theorem}


\begin{proof}
Let $F(\bar{x},t)$ be the defining function above. For each fixed $t\in I$, the fibre
\[
S_t=\{\bar{x}\in\mathbb R^d:F(\bar{x},t)=0\}
\]
is a hypersphere, so $Z(F)$ is the total space of the one-parameter family.

Assume now that $F$ is degenerate in the sense of Definition~\ref{def:degenerate_polynomial}. If both $\gamma$ and $R$ are constant, then the family is concentric and there is nothing to prove. Hence we may assume that $\gamma$ and $R$ are not both constant.

We claim that the coordinate-regular locus $V^\circ$ is dense in
\[
V=Z(F).
\]
First observe that $V$ is smooth. Indeed, on $V$ we have
\[
\sum_{i=1}^d (x_i-a_i(t))^2=R(t)^2>0,
\]
so at least one of the partial derivatives
\[
F_{x_i}=2(x_i-a_i(t))
\]
is nonzero.

Moreover, $V$ is connected; in fact, the map
\[
S^{d-1}\times I\longrightarrow V,\qquad
(u,t)\longmapsto \bigl(\gamma(t)+R(t)u,t\bigr),
\]
is an analytic diffeomorphism. For each $i=1,\ldots,d$, the restriction $F_{x_i}|_V$ is not identically zero. Indeed, for any fixed $t\in I$, at the point
\[
x=\gamma(t)+R(t)e_i
\]
we have
\[
F_{x_i}=2R(t)\neq0.
\] 
It remains to consider $F_t$. We have
\[
F_t= -2\sum_{i=1}^d(x_i-a_i(t))a_i'(t)-2R(t)R'(t).
\]
Suppose that $F_t$ vanished identically on $V$. Writing
\[
x=\gamma(t)+R(t)u, \qquad u\in S^{d-1},
\]
we would obtain
\[
u\cdot\gamma'(t)+R'(t)=0
\]
for every $u\in S^{d-1}$. Replacing $u$ by $-u$ and subtracting the two identities gives
\[
u\cdot\gamma'(t)=0
\]
for every $u\in S^{d-1}$. Hence
\[
\gamma'(t)=0,
\]
and consequently $R'(t)=0$. Hence both $\gamma$ and $R$ would be constant, contrary to our assumption. Thus $F_t|_V$ is not identically zero.

Therefore each coordinate partial derivative restricts to a nonzero real analytic function on the connected real analytic manifold $V$. Its zero set has empty interior, and hence
\[
V^\circ
=
V\setminus
\left(
\bigcup_{i=1}^d\{F_{x_i}=0\}
\cup
\{F_t=0\}
\right)
\]
is dense in $V$.

Since $F$ is locally degenerate at every point of $V^\circ$, Proposition~\ref{prop:boundary-continuation} implies that every local irreducible component of each boundary variety of $F$ is contained in a coordinate slice.


We first treat the case $d\geq 3$. Fix $i\in\{1,\ldots,d\}$, and let $\Gamma$ be a local irreducible component of $\Sigma_i$. By Proposition~\ref{prop:boundary-continuation}, $\Gamma$ is contained in a coordinate slice. Since the variables are $x_1,\ldots,x_d,t$, there are three possibilities:
\[
\Gamma\subset\{x_i=c\},
\]
\[
\Gamma\subset\{x_k=c\}\quad\text{for some }k\neq i,
\]
or
\[
\Gamma\subset\{t=c\}.
\]
The third possibility is impossible by Lemma~\ref{lm: extra boundary variety}. Thus only the first two possibilities remain.

We first consider the case
\[
\Gamma\subset\{x_i=c\}.
\]
On $\Sigma_i$, we have
\[
F_{x_i}=2(x_i-a_i(t))=0,
\]
and hence
\[
x_i=a_i(t).
\]
Therefore, on $\Gamma$, we get
\[
a_i(t)=c.
\]
Since $\Gamma$ is not contained in a slice $t=\mathrm{const}$, the parameter $t$ varies along $\Gamma$. Hence $a_i(t)=c$ locally. By analyticity,
\[
a_i(t)\equiv c \quad \text{on} \quad I.
\]
Thus, if for every $i=1,\ldots,d$ the local components of $\Sigma_i$ are contained in slices of the form $x_i=\mathrm{const}$, then all functions
$a_i(t)$ are constant. Hence the centre curve
\[
\gamma(t)=(a_1(t),\ldots,a_d(t))
\]
is constant, and the family is concentric.

It remains to rule out the second possibility when $d\geq3$. Suppose, for contradiction, that
\[
\Gamma\subset\{x_k=c\} \qquad\text{for some }k\neq i.
\]
On $\Sigma_i$, we have
\[
x_i=a_i(t)
\]
and
\[
\sum_{j\neq i}(x_j-a_j(t))^2=R(t)^2.
\]
Fix a point $q\in \Gamma$, and let $t_0\in I$ be its $t$-coordinate. Since $R(t_0)>0$, the fibre of $\Sigma_i$ over $t_0$ is a genuine $(d-2)$-sphere in the affine hyperplane
\[
\{x_i=a_i(t_0)\}.
\]
Since $d\geq3$, this sphere has positive dimension. If a local component of $\Sigma_i$ were contained in $\{x_k=c\}$, then a nonempty open piece of this positive-dimensional sphere would be contained in the hyperplane
\[
\{x_i=a_i(t_0)\}\cap\{x_k=c\}.
\]
This is impossible, because a positive-dimensional sphere cannot contain a nonempty open piece inside an affine hyperplane. Therefore the second possibility is impossible. Hence, for $d\geq3$, every local component of $\Sigma_i$ must be contained in a slice of the form
\[
x_i=\mathrm{const}.
\]
By the first case, each $a_i(t)$ is constant. Therefore the family is concentric.

We now treat the case $d=2$. In this case
\[
F(x_1,x_2,t)=(x_1-a_1(t))^2+(x_2-a_2(t))^2-R(t)^2.
\]
The boundary varieties are
\[
\Sigma_1=\{F=0,\ F_{x_1}=0\}
\]
and
\[
\Sigma_2=\{F=0,\ F_{x_2}=0\}.
\]
Since
\[
F_{x_1}=2(x_1-a_1(t)),
\]
we have $x_1=a_1(t)$ on $\Sigma_1$. Hence the local components of $\Sigma_1$ are given by
\[
x_1=a_1(t), \qquad x_2=a_2(t)+R(t),\quad \text{and}\quad x_1=a_1(t), \qquad x_2=a_2(t)-R(t).
\]
Similarly, since
\[
F_{x_2}=2(x_2-a_2(t)),
\]
we have $x_2=a_2(t)$ on $\Sigma_2$. Hence the local components of $\Sigma_2$ are given by
\[
x_2=a_2(t), \qquad x_1=a_1(t)+R(t), \quad
\text{and} \quad
x_2=a_2(t), \qquad x_1=a_1(t)-R(t).
\]
Let $\Gamma$ be a local irreducible component of either $\Sigma_1$ or $\Sigma_2$. By Proposition~\ref{prop:boundary-continuation}, $\Gamma$ is contained in a coordinate slice. Since the variables are $x_1,x_2,t$, there are three possibilities:
\[
\Gamma\subset\{x_1=\mathrm{const}\},
\]
\[
\Gamma\subset\{x_2=\mathrm{const}\},
\]
or
\[
\Gamma\subset\{t=\mathrm{const}\}.
\]
The third possibility is impossible by Lemma~\ref{lm: extra boundary variety}. Thus every local component is contained either in a slice $x_1=\mathrm{const}$ or in a slice $x_2=\mathrm{const}$.

We first consider the same-coordinate case. Suppose that the local components of $\Sigma_1$ are contained in slices of the form
\[
x_1=\mathrm{const},
\]
and the local components of $\Sigma_2$ are contained in slices of the form
\[
x_2=\mathrm{const}.
\]
Since $x_1=a_1(t)$ on $\Sigma_1$, it follows that $a_1(t)$ is locally constant. By analyticity, $a_1(t)$ is constant on $I$. Similarly, since $x_2=a_2(t)$ on $\Sigma_2$, it follows that $a_2(t)$ is constant on $I$. Therefore the centre curve
\[
\gamma(t)=(a_1(t),a_2(t))
\]
is constant, and the family is concentric.

It remains to analyse the mixed case. If both $a_1(t)$ and $a_2(t)$ are constant, then we are already in the concentric case. Hence we may assume that at least one of them is not constant.

Suppose first that $a_1(t)$ is not constant. Then no local component of $\Sigma_1$ can be contained in a slice
\[
x_1=\mathrm{const},
\]
because on $\Sigma_1$ we have $x_1=a_1(t)$, and the parameter $t$ varies along each local component. Therefore both local components of $\Sigma_1$ must be contained in slices of the form
\[
x_2=\mathrm{const}.
\]
Using the explicit description of the two local components of $\Sigma_1$, we obtain
\[
a_2(t)+R(t)=c_+
\]
and
\[
a_2(t)-R(t)=c_-
\]
for some constants $c_+$ and $c_-$. Adding and subtracting these two identities gives
\[
a_2(t)=\frac{c_++c_-}{2}
\quad \text{and} \quad
R(t)=\frac{c_+-c_-}{2}.
\]
Thus $a_2(t)$ and $R(t)$ are constant. Hence the centre curve is contained in the line
\[
x_2=\mathrm{const},
\]
and the radius is constant.

The remaining case is symmetric. If $a_2(t)$ is not constant, then no local component of $\Sigma_2$ can be contained in a slice
\[
x_2=\mathrm{const}.
\]
Therefore both local components of $\Sigma_2$ must be contained in slices of the form
\[
x_1=\mathrm{const}.
\]
Using the explicit description of the two local components of $\Sigma_2$, we obtain
\[
a_1(t)+R(t)=c_+
\]
and
\[
a_1(t)-R(t)=c_-
\]
for some constants $c_+$ and $c_-$. Hence $a_1(t)$ and $R(t)$ are constant. Therefore the centre curve is contained in the line
\[
x_1=\mathrm{const},
\]
and the radius is constant.

We conclude that, in dimension $2$, either the family is concentric, or the centre curve is contained in a line parallel to one of the coordinate axes and the radius is constant.

We now prove the converse direction. First assume that the family is concentric. Then there are constants
$c_1,\ldots,c_d$ such that
\[
a_i(t)=c_i \qquad\text{for all }i=1,\ldots,d.
\]
Hence
\[
F(x_1,\ldots,x_d,t)=\sum_{i=1}^d (x_i-c_i)^2-R(t)^2.
\]
Let $p=(x_1^0,\ldots,x_d^0,t_0)$ be a coordinate regular point of $Z(F)$. Then
\[
x_i^0\neq c_i \qquad\text{for all }i=1,\ldots,d,
\]
and also
\[
F_t(p)=-2R(t_0)R'(t_0)\neq 0.
\]
Hence $R(t_0)$ and $R'(t_0)$ are non zero. Therefore the one-variable functions
\[
f_i(x_i)=(x_i-c_i)^2, \qquad i=1,\ldots,d,
\]
and
\[
f_{d+1}(t)=-R(t)^2
\]
are local analytic isomorphisms near $x_i^0$ and $t_0$, respectively. Since their values at the corresponding coordinates of $p$ sum to $0$, after subtracting those values and shrinking the domains we may assume that they are analytic isomorphisms onto a common interval $J$ containing $0$. Moreover, near $p$, we have
\[
F(x_1,\ldots,x_d,t)=0
\]
if and only if
\[
f_1(x_1)+\cdots+f_d(x_d)+f_{d+1}(t)=0.
\]
Thus $F$ is locally degenerate at $p$. Since $p$ was arbitrary, $F$ is degenerate.

It remains to consider the additional case in dimension $2$. After relabelling the coordinates, we may assume that
\[
a_2(t)=c \qquad\text{and}\qquad R(t)=R_0>0.
\]
Thus
\[
F(x_1,x_2,t)=(x_1-a_1(t))^2+(x_2-c)^2-R_0^2.
\]
Let $p=(x_1^0,x_2^0,t_0)$ be a coordinate regular point of $Z(F)$. Near $p$, one can choose a sign $\varepsilon\in\{+1,-1\}$ such that
\[
F(x_1,x_2,t)=0
\]
is equivalent to
\[
-x_1+\varepsilon\sqrt{R_0^2-(x_2-c)^2}+a_1(t)=0.
\]
The three functions
\[
f_1(x_1)=-x_1, \qquad f_2(x_2)=\varepsilon\sqrt{R_0^2-(x_2-c)^2}, \qquad f_3(t)=a_1(t)
\]
have nonzero derivatives at $x_1^0,x_2^0,t_0$, respectively, because $p$ is coordinate regular. Hence they are local analytic isomorphisms. Since their values at $(x_1^0,x_2^0,t_0)$ sum to $0$, after subtracting those values and shrinking the domains we may assume that they are analytic isomorphisms onto a common interval $J$ containing $0$. Therefore, near $p$,
\[
F(x_1,x_2,t)=0
\]
if and only if
\[
f_1(x_1)+f_2(x_2)+f_3(t)=0.
\]
Thus $F$ is locally degenerate at $p$. Since $p$ was arbitrary, $F$ is degenerate. The case $a_1(t)=c$ and $R(t)=R_0$ is symmetric.
\end{proof}

\section{Proof of Theorem~\ref{thm: pencil of spheres}}  \label{sec: pencil of spheres}

Let $\mathcal P=\{S_t:t\in I\}$ be a regular one-parameter polynomial family of hyperspheres in $\mathbb R^d$. By definition, $\mathcal P$ is represented by a polynomial equation
\[
F(\bar{x},t):=\sum_{i=1}^d (x_i-a_i(t))^2-R^2(t)=0,
\]
where
\[
a_1(t),\dots,a_d(t),R(t)\in\mathbb R[t].
\]
We denote by
\[
Q_{\mathcal P}=\{(\bar x,t)\in\mathbb R^d\times I:F(\bar x,t)=0\}
\]
the corresponding incidence hypersurface. Since $\mathcal P$ is regular, $Q_{\mathcal P}$ is smooth, and the projection of $Q_{\mathcal P}$ onto any $d$ of the $d+1$ coordinates is finite-to-one.

Our goal is to classify those pencils which determine a near-maximal number of incidences with a Cartesian grid. Let
\[
A_1,\dots,A_d\subset \mathbb R, \qquad T_0\subset I
\]
be finite sets with
\[
|A_1|=\cdots=|A_d|=|T_0|=n.
\]
A tuple
\[
(x_1,\ldots,x_d,t)\in A_1\times\cdots\times A_d\times T_0
\]
will be called a $(d+1)$-rich tuple if the point
\[
\bar{x}=(x_1,\ldots,x_d)
\]
lies on the hypersphere $S_t$. Equivalently,
\[
F(x_1,\ldots,x_d,t)=0.
\]
Thus the set of all $(d+1)$-rich tuples is precisely
\[
Q_{\mathcal P}\cap(A_1\times\cdots\times A_d\times T_0).
\]
Assume that $Q_{\mathcal P}$ intersects such a Cartesian grid in a near-maximal number of points, in the sense that
\[
|Q_{\mathcal P}\cap(A_1\times\cdots\times A_d\times T_0)|= \Omega(n^{d-\eta}).
\]
We will prove that such families are necessarily highly structured: they consist either of concentric hyperspheres in any dimension, or, in the special case $d=2$, of circles with fixed radius whose centres lie on a line parallel to one of the coordinate axes.

We shall use the following Elekes-Szab\'{o} theorem, due to Chernikov, Peterzil and Starchenko~\cite[Theorem~A]{CPS}.
\begin{theorem}\label{thm:ES-general}
Let $s\geq3$, and let $Q\subset\mathbb R^s$ be a semialgebraic set of description complexity at most $D$. Assume that the projection of $Q$ to any $s-1$ coordinates is finite-to-one. Then exactly one of the following holds:
\begin{enumerate}
\item \emph{(Incidence bound)} There exist constants $c=c(s,D)>0$ and $\eta_0=\eta_0(s)>0$ such that, for all finite sets $A_i\subset\mathbb R$ with $|A_i|=n$,
\[
|Q\cap(A_1\times\cdots\times A_s)|\leq c\,n^{s-1-\eta_0}.
\]
\item \emph{(Group structure)} There exist nonempty open sets $U_i\subset\mathbb R$, an open set $V\subset\mathbb R$ containing $0$, and analytic bijections $\phi_i:U_i\to V$ with analytic inverses such that, for all $x_i\in U_i$,
\[
(x_1,\ldots,x_s)\in Q
\quad\Longleftrightarrow\quad
\phi_1(x_1)+\cdots+\phi_s(x_s)=0.
\]
\end{enumerate}
\end{theorem}
In our application we take $s=d+1$ and
\[
Q_{\mathcal P}
=
\{(\bar x,t)\in\mathbb R^d\times I:F(\bar x,t)=0\},
\qquad
F(\bar x,t)=\sum_{i=1}^d (x_i-a_i(t))^2-R^2(t).
\]
Since $\mathcal P$ is a polynomial family and $I$ is an interval, $Q_{\mathcal P}$ is semialgebraic, with description complexity bounded in terms of $d$ and the degree of $F$. The regularity assumption says precisely that the projection of $Q_{\mathcal P}$ onto any $d$ of the $d+1$ coordinates is finite-to-one. Thus Theorem~\ref{thm:ES-general} applies directly to $Q_{\mathcal P}$.

Recall that $\eta=\eta_0/2$. Hence our assumption gives
\[
|Q_{\mathcal P}\cap (A_1\times\cdots\times A_d\times T_0)| = \Omega(n^{d-\eta_0/2}).
\]
On the other hand, the incidence-bound alternative in Theorem~\ref{thm:ES-general} would give
\[
|Q_{\mathcal P}\cap (A_1\times\cdots\times A_d\times T_0)|= O(n^{d-\eta_0}).
\]
Since
\[
n^{d-\eta_0/2}\gg n^{d-\eta_0},
\]
these two estimates are incompatible for sufficiently large $n$. Therefore the incidence-bound alternative cannot hold. Hence the group-structure alternative holds, so there is a nonempty relatively open subset $W$ of $Q_{\mathcal P}$ on which
\[
F(\bar x,t)=0 \quad\Longleftrightarrow\quad \phi_1(x_1)+\cdots+\phi_d(x_d)+\phi_{d+1}(t)=0
\]
for one-variable analytic local isomorphisms $\phi_i$. Since $Q_{\mathcal P}$ is smooth and the two equations define the same hypersurface on $W$, their gradients are proportional there. As $\phi_i'\neq0$ for every $i$, every point of $W$ is coordinate regular. Thus $F$ is locally degenerate on $W$. Moreover, as observed in the proof of Theorem~\ref{thm:analytic-hypersphere-degeneracy}, $Q_{\mathcal P}$ is connected. Hence Lemma~\ref{lem:propagation-degeneracy}, applied on the connected hypersurface $Q_{\mathcal P}$, implies that $F$ is locally degenerate at every coordinate-regular point of $Q_{\mathcal P}$.

Therefore Theorem~\ref{thm:analytic-hypersphere-degeneracy} applies. It follows that the family $\mathcal P$ is concentric in dimensions $d\geq3$. In dimension $d=2$, the family is either concentric, or consists of circles of fixed radius whose centres lie on a line parallel to one of the coordinate axes. This proves Theorem~\ref{thm: pencil of spheres}.

\begin{remark}[Sharpness of Theorem~\ref{thm: pencil of spheres}]
The two alternatives in Theorem~\ref{thm: pencil of spheres} are sharp. In the concentric case, for each $i=1,\ldots,d$, take the coordinate
hyperplanes $x_i=\sqrt{k}$, $1\leq k\leq n$, and take the hyperspheres $\|x\|^2=m$, where $m\in\{d,2d,\ldots,dn\}$. Then a positive proportion of the grid points satisfy
\[
k_1+\cdots+k_d\equiv 0 \pmod d,
\]
and hence give $\Theta(n^d)$ rich points. In the planar exceptional case, the family
\[
(x-t)^2+y^2=1
\]
together with the choices $x=i\varepsilon$, $y=\sqrt{1-(j\varepsilon)^2}$, and $t=(i+j)\varepsilon$, for $1\leq i,j\leq n$ and $i+j\leq n+1$, gives $\Theta(n^2)$ rich points.
\end{remark}

\section{Proof of Theorem~\ref{thm: pin distance}}\label{sec: pinned distance}

In this section, we prove the higher-dimensional pinned distance result stated in Theorem~\ref{thm: pin distance}. Let $p_1,\dots,p_{d+1}\in \mathbb{R}^d$ be fixed points, and for each $i=1,\dots,d+1$, let
\[
\mathcal{S}_i=\left\{S(p_i,r): r \in \mathcal{R}_i \right\}
\]
be a family of hyperspheres centred at $p_i$, where $|\mathcal{R}_i|=n$. A point $q\in \mathbb{R}^d$ is called $(d+1)$-rich if, for every $i=1,\dots,d+1$, there exists $r_i\in \mathcal{R}_i$ such that
\[
q\in S(p_i,r_i).
\]
Equivalently,
\[
\|q-p_i\|^2=r_i^2 \qquad \text{for all} \qquad i=1,\dots,d+1.
\]
Set
\[
A_i=\{r^2:r\in \mathcal{R}_i\}.
\]
Then every $(d+1)$-rich point $q$ gives a tuple
\[
(t_1,\dots,t_{d+1})\in A_1\times\cdots\times A_{d+1}, \qquad t_i=\|q-p_i\|^2.
\]
The key point is that the variables $t_1,\dots,t_{d+1}$ satisfy a polynomial relation obtained by eliminating $q$. The following proposition characterizes the degeneracy of this polynomial relation in terms of the affine geometry of the configuration $p_1,\dots,p_{d+1}$.

\begin{proposition}\label{prop:distance-degenerate}
Let $p_1,\dots, p_{d+1} \in \R^d$ be affinely independent and define
\[
t_i=\|q - p_i\|^2,\quad q=(x_1,\dots,x_d)\in \R^d.
\]
Let $F(t_1,\dots, t_{d+1})=0$ be the polynomial relation obtained by eliminating $q$. Then~$F$ is non-degenerate in the sense of Definition~\ref{def:degenerate_polynomial}.
\end{proposition}

\begin{proof}
Without loss of generality, we may translate the configuration and assume that $p_{d+1}$ is the origin. Then the vectors
\[
p_i=p_i-p_{d+1}, \qquad i=1,\dots,d,
\]
are linearly independent.
	
Let $P$ be the $d\times d$ matrix whose $i$-th row is $p_i$. Since the vectors $p_1,\dots,p_d$ are linearly independent, the matrix $P$ is invertible. For $i=1,\dots,d$, consider
\[
t_i-t_{d+1}= \|q-p_i\|^2-\|q\|^2, \qquad q=(x_1,\dots,x_d)\in\mathbb R^d.
\]
Expanding the right-hand side gives
\[
t_i-t_{d+1}=\|p_i\|^2-2p_i\cdot q.
\]
Define
\[
u_i=t_i-t_{d+1}-\|p_i\|^2, \qquad i=1,\dots,d,
\]
and write
\[
u=(u_1,\dots,u_d)^T.
\]
Then
\[
u_i=-2p_i\cdot q,
\]
and hence
\[
u=-2Pq.
\]
Since $P$ is invertible, $q$ can be computed uniquely in terms of $u$:
\[
q=-\frac{1}{2}P^{-1}u.
\]
Using $t_{d+1}=\|q\|^2=q^Tq$, we obtain
\[
t_{d+1}=\left(-\frac{1}{2}P^{-1}u\right)^T \left(-\frac{1}{2}P^{-1}u\right).
\]
Thus
\[
t_{d+1}=\frac{1}{4}u^TP^{-T}P^{-1}u.
\]
Set
\[
B=P^{-T}P^{-1}.
\]
Then the polynomial relation among $t_1,\dots,t_{d+1}$ is
\begin{equation}\label{eq: definite polynomial}
F(t_1,\dots,t_{d+1})=t_{d+1}-\frac{1}{4}u^TBu.
\end{equation}
Since $P$ is invertible, the matrix $B$ is symmetric positive definite. Indeed,
\[
u^TBu=(P^{-1}u)^T(P^{-1}u)=\|P^{-1}u\|^2.
\]
\textbf{Claim $Z(F)$ is smooth and irreducible:}

To do this, introduce the new variables
\[
z=t_{d+1},\qquad u_i=t_i-t_{d+1}-\|p_i\|^2, \qquad i=1,\dots,d.
\]
This is an invertible affine change of coordinates. In the variables $(u_1,\dots,u_d,z)$, the polynomial $F$ has the form
\[
F(u,z)=z-\frac14 u^TBu.
\]
Hence
\[
Z(F)=\left\{(u,z):z=\frac14 u^TBu\right\}
\]
is the graph of a polynomial function over $\R^d$. Therefore $Z(F)$ is irreducible. Moreover,
\[
\frac{\partial F}{\partial z}=1,
\]
so $Z(F)$ is smooth.

We now prove that $F$ is non-degenerate in the sense of Definition~\ref{def:degenerate_polynomial}. We argue using the boundary variety criterion. Consider the boundary variety
\[
\Sigma_1=\left\{t\in Z(F): \frac{\partial F}{\partial t_1}(t)=0 \right\}.
\]
Since
\[
u_i=t_i-t_{d+1}-\|p_i\|^2,
\]
we have
\[
\frac{\partial u}{\partial t_1}=e_1=(1,0,\dots,0)^T.
\]
Recall that
\[
F(t)=t_{d+1}-\frac{1}{4}u^TBu.
\]
Since $B$ is symmetric, the chain rule gives
\[
\frac{\partial}{\partial t_1}(u^TBu)=2e_1^TBu.
\]
Therefore
\[
\frac{\partial F}{\partial t_1}=-\frac{1}{2}e_1^TBu=-\frac{1}{2}(Bu)_1.
\]
Thus
\[
\frac{\partial F}{\partial t_1}=0\quad\Longleftrightarrow \quad(Bu)_1=0.
\]
Since $u=-2Pq$, we have
\[
Bu=P^{-T}P^{-1}(-2Pq)=-2P^{-T}q.
\]
Therefore
\begin{equation}\label{eq: partial derivative}
\frac{\partial F}{\partial t_1}=0\quad\Longleftrightarrow\quad(P^{-T}q)_1=0.
\end{equation}
Now define the distance map
\[
\Phi:\R^d\to \R^{d+1},\qquad \Phi(q)=\left(\|q-p_1\|^2,\dots,\|q-p_{d+1}\|^2\right).
\]
By construction, $Z(F)$ is the image of $\Phi$. Moreover, $\Phi$ is an analytic diffeomorphism from $\R^d$ onto $Z(F)$: indeed, in the affine coordinates above its inverse is given by $q=-\frac12P^{-1}u$. Define
\[
H=\left\{q\in\mathbb R^d:(P^{-T}q)_1=0\right\}.
\]
This is a hyperplane in the original $q$-space. We claim that
\[
\Sigma_1=\Phi(H).
\]
Indeed, let $t\in\Sigma_1$. Since $t\in Z(F)$, there exists $q\in\R^d$ such that
\[
\Phi(q)=t.
\]
Since $F_{t_1}(t)=0$, equation~\eqref{eq: partial derivative} gives
\[
(P^{-T}q)_1=0.
\]
Hence $q\in H$, and therefore $t\in\Phi(H)$. This proves
\[
\Sigma_1\subset\Phi(H).
\]
Conversely, let $q\in H$, and put $t=\Phi(q)$. Then $t\in Z(F)$. Moreover, since $q\in H$, we have
\[
(P^{-T}q)_1=0.
\]
Again by equation~\eqref{eq: partial derivative}, this implies
\[
F_{t_1}(t)=0.
\]
Therefore $t\in\Sigma_1$. Hence
\[
\Phi(H)\subset\Sigma_1.
\]
Thus
\[
\Sigma_1=\Phi(H).
\]
We now show that no local component of $\Sigma_1$ is contained in a coordinate slice of the form $t_j=c$. Fix $j\in\{1,\dots,d+1\}$. On $H$, the $j$-th coordinate of the distance map is
\[
t_j(q)=\|q-p_j\|^2,
\]
where $p_{d+1}=0$. We claim that this function is not locally constant on $H$. Indeed, since $d\geq2$, the hyperplane $H$ has positive dimension. Choose $q\in H$ and a nonzero vector $v$ tangent to $H$. Then $q+sv\in H$ for all sufficiently small $s\in\mathbb R$. Along this line we have
\[
t_j(q+sv)=\|q+sv-p_j\|^2=\|q-p_j\|^2+2s\,v\cdot(q-p_j)+s^2\|v\|^2.
\]
Since $v\neq0$, we have $\|v\|^2>0$, and therefore this is a nonconstant quadratic polynomial in $s$. Hence $t_j$ is not locally constant on $H$. It follows that no coordinate~$t_j$ is locally constant on $\Phi(H)=\Sigma_1$. Therefore no local irreducible component of $\Sigma_1$ is contained in a coordinate slice
\[
t_j=c.
\]
Suppose, for contradiction, that $F$ is degenerate in the sense of Definition~\ref{def:degenerate_polynomial}. We first verify that its coordinate-regular locus is dense in $Z(F)$.

As observed above, in the coordinates $(u,z)$ the hypersurface $Z(F)$
is the graph
\[
z=\frac14 u^TBu,
\]
and hence is a connected smooth real analytic manifold. Moreover, for
$j=1,\ldots,d$,
\[
\frac{\partial F}{\partial t_j}
=
-\frac12(Bu)_j,
\]
while
\[
\frac{\partial F}{\partial t_{d+1}}
=
1+\frac12\sum_{j=1}^d(Bu)_j.
\]
Since $B$ is invertible and $u$ ranges over all of $\mathbb R^d$ on
$Z(F)$, none of these functions vanishes identically on $Z(F)$.
Therefore the zero set of each coordinate partial derivative has empty
interior in $Z(F)$, and hence the coordinate-regular locus is dense in
$Z(F)$.

Proposition~\ref{prop:boundary-continuation} now implies that every
local irreducible component of $\Sigma_1$ is contained in a coordinate
slice. This contradicts what we have just proved. Therefore $F$ is
non-degenerate.

\end{proof}

Now we are ready to prove Theorem~\ref{thm: pin distance}.
\begin{proof}
Assume, for contradiction, that $p_1,\dots,p_{d+1}$ are affinely independent. For each $i$, set
\[
A_i=\{r^2:r\in \mathcal{R}_i\}.
\]
A $(d+1)$-rich point $q$ determines a tuple
\[
(t_1,\dots,t_{d+1})\in A_1\times\cdots\times A_{d+1}, \qquad t_i=\|q-p_i\|^2,
\]
satisfying the polynomial relation $F(t_1,\dots,t_{d+1})=0$ in \eqref{eq: definite polynomial}. Moreover, distinct $(d+1)$-rich points determine distinct tuples. Indeed, suppose that $q$ and $q'$ determine the same tuple. Then, for
each $i=1,\ldots,d$,
\[
\|q-p_i\|^2-\|q-p_{d+1}\|^2=\|q'-p_i\|^2-\|q'-p_{d+1}\|^2.
\]
Expanding and cancelling gives
\[
(p_i-p_{d+1})\cdot(q-q')=0, \qquad i=1,\ldots,d.
\]
Since $p_1,\ldots,p_{d+1}$ are affinely independent, the vectors
\[
p_1-p_{d+1},\ldots,p_d-p_{d+1}
\]
form a basis of $\mathbb R^d$. Hence $q=q'$. 
Thus the map from $(d+1)$-rich points to the corresponding tuples is
injective.

We now apply Theorem~\ref{thm:ES-general} to the semialgebraic relation $Q=Z(F)\subset\mathbb R^{d+1}$. It remains only to verify the projection condition. In fact, every projection of $Z(F)$ onto $d$ of the $d+1$ coordinates has fibres of size at most two.

First fix $j\in\{1,\ldots,d\}$ and fix all coordinates except $t_j$. Since
\[
u_j=t_j-t_{d+1}-\|p_j\|^2,
\]
while the other components of $u$ are fixed, the polynomial
\[
F=t_{d+1}-\frac14 u^TBu
\]
is a quadratic polynomial in $t_j$ whose quadratic coefficient is
\[
-\frac14 B_{jj}.
\]
Since $B$ is positive definite, $B_{jj}>0$. Thus, for fixed values of the remaining coordinates, the equation $F=0$ has at most two solutions in $t_j$.

It remains to consider the projection which forgets $t_{d+1}$. Fix $t_1,\ldots,t_d$ and set
\[
a_i=t_i-\|p_i\|^2,\qquad a=(a_1,\ldots,a_d)^T,
\]
and let
\[
\mathbf{1}=(1,\ldots,1)^T.
\]
Writing $z=t_{d+1}$, we have
\[
u=a-z\mathbf{1}.
\]
Hence
\[
F
=
z-\frac14(a-z\mathbf{1})^TB(a-z\mathbf{1}),
\]
which is a quadratic polynomial in $z$ with quadratic coefficient
\[
-\frac14\,\mathbf{1}^TB\mathbf{1}.
\]
Since $B$ is positive definite and $\mathbf{1}\neq0$, we have
\[
\mathbf{1}^TB\mathbf{1}>0.
\]
Thus this equation also has at most two solutions in $z$.

Consequently, the projection of $Z(F)$ onto any $d$ of the $d+1$ coordinates is at most two-to-one. In particular, the projection condition in Theorem~\ref{thm:ES-general} is satisfied.
Therefore $Q=Z(F)$ satisfies the hypotheses of Theorem~\ref{thm:ES-general} for $s=d+1$. Consequently, if $N$ denotes the number of $(d+1)$-rich points, then
\[
\bigl|Z(F)\cap (A_1\times\cdots\times A_{d+1})\bigr| \geq N.
\]
Now suppose that the number of $(d+1)$-rich points is $\Omega(n^{d-\eta})$, where $\eta=\eta_0/2$. By the injectivity established above,
\[
\bigl|Z(F)\cap (A_1\times\cdots\times A_{d+1})\bigr| =\Omega(n^{d-\eta_0/2}).
\]
This contradicts the incidence-bound alternative in Theorem~\ref{thm:ES-general}, which gives an upper bound $O(n^{d-\eta_0})$. Hence the group-structure alternative holds. Under our assumption that $p_1,\ldots,p_{d+1}$ are affinely independent, the proof of Proposition~\ref{prop:distance-degenerate} shows that $Z(F)$ is a connected smooth real analytic hypersurface. On the open set given by the group-structure alternative, the equation $F=0$ and the local additive equation define the same smooth hypersurface. Their gradients are therefore proportional there. Since the functions in the additive representation are analytic local isomorphisms, all of their derivatives are nonzero. Hence, after shrinking the open set if necessary, it is contained in the coordinate-regular locus of $Z(F)$, and $F$ is locally degenerate there.

Lemma~\ref{lem:propagation-degeneracy}, applied in $d+1$ variables, therefore implies that $F$ is locally degenerate at every coordinate-regular point of $Z(F)$. Hence $F$ is degenerate in the sense of Definition~\ref{def:degenerate_polynomial}. This contradicts Proposition~\ref{prop:distance-degenerate}, which shows that $F$ is non-degenerate when $p_1,\ldots,p_{d+1}$ are affinely independent.

\end{proof}


\section*{Funding}

This work was supported by the Austrian Science Fund (FWF) [grant number J 4741-N].

\section*{Acknowledgements}

The author would also like to acknowledge the assistance of ChatGPT (OpenAI) in discussing parts of the manuscript and in improving its presentation. All mathematical statements and proofs were checked and remain the responsibility of the author.

\end{document}